\documentclass[12pt]{amsart}

\usepackage[left=3cm,top=3.0cm,right=3cm,bottom=2.6cm]{geometry}

\usepackage[ansinew]{inputenc}
\usepackage{amscd,amsfonts,amsmath,amssymb,amsthm,enumerate,epsfig,float,graphics,graphicx,pst-all,yhmath}
\newcommand{\inte}{\operatorname*{int}}

\newcommand{\bd}{\operatorname*{bd}}

\newcommand{\Rt}{\mathbb{R}^3}
\newcommand{\Rn}{\mathbb{R}^{n}}

\newtheorem{lemma}{Lemma}

\newtheorem{theorem}{Theorem}

 \title{Characterization of the sphere by means of congruent support cones}

\author{E. Morales Amaya} 
\address{Facultad de Matem\'aticas-Acapulco, Universidad Aut\'onoma de Guerrero, Carlos E. Adame 54, Col. Garita C.P. 39650, Acapulco, Guerrero. Mexico.}
\email{emoralesamaya@gmail.com}

\thanks{I would like to thank to Luis Montejano for the interesting conversations and suggestions regarding the use of topology of fibre bundles in convex geometry. This research was supported by the National Council
of Humanities Sciences and Technology (CONAHCyT) of Mexico, SNI 21120}

\begin{document}

\begin{abstract}
Let $M$ be a convex body and let $K$ be a closed convex surface $K$ both contained in the Euclidean space $\mathbb{E}^3$. What can we say about $M$ if $K$ encloses $M$ and if from all the points in $K$ the body $M$ looks the same? In this work we are going to present a result which claims that if for every two support cones $C_x$, $C_y$ of $M$, with apexes $x,y \in  K$, respectively, there exists $\Phi$  in the  semi direct product of the orthogonal group $O(3)$ and $\mathbb{E}^3$ such that
$$C_y=\Phi(C_x),$$
and this can be done in a continuous way, then $M$ is a sphere.
\end{abstract}
\maketitle
\textbf{Introduction.} 
Let $\mathbb{E}^{n+1}$ be the Euclidean space of dimension $n+1$ endowed with the usual inner pro\-duct $\langle \cdot, \cdot\rangle : \mathbb{E}^{n+1} \times \mathbb{E}^{n+1} \rightarrow \mathbb{R}$.  
For $n \geq 2$, we denote by $A(n+1)$ the affine group of $\mathbb{E}^{n+1}$ and by $O(n+1)$ the \textit{orthogonal group}, i.e., the set of all the isometries of $\mathbb{E}^{n+1}$ that fix the origin (equivalently, the collection of all the orthogonal matrices, where an $(n+1)\times (n+1)$ matrix  $D$ is said to be orthogonal if $D^{t}D$ is the identity matrix).  
    
Let $M\subset \mathbb{E}^{n}$ be a convex body. Given a point $x \in \mathbb{E}^{n+1} \backslash M$ we denote the cone generated by $M$ with apex $x$ by $S_x$, that is, $S_x := \{x + \lambda(y - x) : y \in M, \lambda \geq  0\}$, and by $C_x$ the boundary of $S_x$, in other words, $C_x$ is the support cone of $M$ from the point $x$. Let $x,y\in \mathbb{E}^{n+1}$. The convex cones $C_x$, $C_y$ are said to be \textit{affinely congruent} if there exists $\Phi \in A(n+1)$ of the form $\Phi=a+\Omega$, $a\in \mathbb{E}^{n+1}$ and $\Omega \in O(n+1)$, such that $\Phi(C_x)=C_y$.   

Let $M,K\subset \mathbb{E}^{3}$ be convex bodies, $M\subset \inte K$. Suppose that all the support cones of $M$ with apexes in $\bd K$ are affinely congruent. Furthermore, assume that if $x_0,y_0,x,y\in \bd K$,   $\Omega_1, \Omega_2\in O(3)$ and $a_1,a_2\in \mathbb{E}^{3}$ are such that $C_x=\Phi_1(C_{x_0})$ and $C_y=\Phi_2(C_{y_0}),$
where $\Phi_1:=a_1+\Omega_1$, $\Phi_2:=a_2+\Omega_2$ , and if $y_0$ \textit{is close to} $x_0$ and $y$ \textit{is close to} $x$ (in $\mathbb{E}^3$) then $\Phi_1$ \textit{is close to} $\Phi_2$ (Since the affine group $A(n+1)$ is composed of matrices it has a natural topology, namely the subspace topology from $\mathbb{E}^{(n+1)^2}$). If this is so we say that the support cones of $M$, with apexes in $\bd K$, are affinely congruent in a \textit{continuous way}.

Our main  result in this work is the following theorem.
\begin{theorem}\label{marisol}
Let $M,K\subset \mathbb{E}^{3}$ be convex bodies, $M\subset \inte K$. Suppose that all the support cones of $M$ with apexes in $\bd K$ are affinely congruent in a continuous way. Then $M$ is a ball.
\end{theorem}

 
           
\textbf{Motivation.} Let us consider the following problem.

\textbf{Problem 1.} \textit{Determine properties of convex bodies $K$ in $\mathbb{E}^n$ imposing conditions on the  support cones of $K$ whose apexes are in a fixed hyperplane.}

Let $M\subset \mathbb{E}^n$ be a convex body, $n\geq 3$. Suppose $H$ is a hyperplane not intersecting the convex body $M$. As a first example of a result related to Problem 1 we have: if, for every $x \in H$, the cone $C_x$ is ellipsoidal (i.e., it has a $(n-1)$-section which is an ellipsoid), then $M$ is an $n$-ellipsoid. 

Notice that if the cone $C_x$ is ellipsoidal, then every bounded  $(n-1)$-section  is an ellipsoid (see, for example, Lemma 1 in \cite{egj}). Let $\phi:\mathbb{P}^n \rightarrow \mathbb{P}^n$ be a projective isomorphism such that $\phi(H)$ is the hyperplane at the infinite $H_{ \infty}$. Since every bounded  $(n-1)$-section of $C_x$ is an ellipsoid it follows that every bounded  $(n-1)$-section of the cylinder $\phi(C_x)$ is an ellipsoid. In particular, the section given by an hyperplane perpendicular to the lines defined by $\phi(C_x)$. That is, the orthogonal projection of $\phi(M)$, in the direction parallel to the lines defined by $\phi(C_x)$, is an ellipsoid. Thus the above condition is equivalent that all the orthogonal projections of $\phi(M)$ are ellipsoids. It is well known that such condition implies that the body $\phi(M)$ is an ellipsoid (see Lemma 2 in \cite{Chakerian}). Hence $M$ is an ellipsoid.

On the other hand, as another example, in 1959, Marchaud \cite{Marchaud} supposed that $S_x\cap \bd M$ is flat, for every $x \in H$, and proved that $M$ must be an ellipsoid (The case $H\cap K\not=\emptyset$ was considered in \cite{MM2}). 
 
Naturally, we can replace in the Problem 1 the condition that the set of apexes is a subset of a hyperplane by the condition that they are in a closed hypersurface $S$ which contains in its interior the body $M$. In particular, we can assume that $S$ is the boundary of a convex body $K \subset \Rn$ such that $M \subset \inte K$. An interesting example of this type is the well known \cite{Matsuura}: 
 
\textbf{Matsuura's Theorem}. Let $M \subset \Rt$ be a convex body and let
$S$ be the boundary of a convex body $K \subset \Rn$ such that $M \subset \inte K$. If the support cone $C_x$ of $M$ is a right circular cone, for every $x\in S$, then $M$ is a Euclidean ball.

(The theorem of Matsuura is slightly more general but for our purposes the above version is good enough).   

By virtue of the observation above, we are interested in the following problem relative to support cones of a convex body in $\mathbb{E}^n$. 

\textbf{Problem 2.} \textit{Given a subgroup $G$ of the general linear group $GL(R, n)$, with $n \geq 3$ and a hypersurface $S$ which is the image of an embedding of $\mathbb{S}^{n-1}$, determine the convex bodies $K\subset \mathbb{E}^n$ such that for any two support cones $\Lambda, \Gamma$ of $K$, with apexes in $S$, there exists an element $\Phi \in G$ with $\Phi(\Lambda)=\Gamma$}.

Next we mention  some related results. Let $M\subset \mathbb{E}^n$, $n \geq 3$, be a convex body and let $S$ be a hypersurface which is the image of an embedding of the sphere $\mathbb{S}^{n-1}$ such that $M\subset \inte S$. It was proven in \cite{bg} that if, for every $x \in S$, the cone $C_x$ is ellipsoidal, then $M$ is an $n$-ellipsoid. On the other hand, it was proven in \cite{banda1} that if $M$ is strictly convex and, for every $x \in S$, there exists $y\in S$ with the property that $C_x$ and $C_y$ differ by a central symmetry, then $M$ and $S$ are centrally symmetric and concentric.

Very recently in \cite{Myr}  it has been proved that if $P$ and $Q$ are polytopes contained in the interior of the ball $B_r(n)$, $n\geq 3$, and if from every point in the sphere $r\mathbb{S}^{n-1}$ the support cones of $P$ and $Q$ are congruent, then $P=Q$. 

For more results similar or related to this, the interested reader may consult the re\-fe\-ren\-ces \cite{bg}, \cite{IJED1}, \cite{IJED2}, \cite{banda2}, \cite{mackey}.

Our proof of Theorem \ref{marisol} requires some notions and results from Topology. Three important elements in our proof of Theorem \ref{marisol} are the notions of a \textit{principal fibre bundle}, a \textit{section} of a principal fibre bundle and \textit{field of congruent bodies}. In fact, we will use the \textit{non existence of a section} of the principal fibre bundle $\tau: \mathbb{S}^2 \rightarrow O(3)$ (see the Apendix). The Topology of fibre bundles in convexity has proven to be a remarkable tool in relation to the so-called Banach Isometric Conjecture relative to the determination of  a convex body by means congruent sections passing through a fixed point \cite{mani}, \cite{monte}. 
\section{Preliminaries.} 
We take an orthogonal system of coordinates $(x_1,...,x_{n+1} )$ for  $\mathbb{E}^{n+1}$ such that the origin $O$ is at the center of the unit sphere 
$\mathbb{S}^{n}=\{x\in \mathbb{E}^{n+1}: \|x\| = 1\}$. For $u\in \mathbb{S}^{n}$, we denote by $H^ {+}(u)$  the half-space $\{x \in \mathbb{E} ^{n}: x\cdot u \leq 0\}$ with unit normal vector $u$, by $H(u)$ its boundary hyperplane $\{x \in \mathbb{E} ^{n}: x\cdot u = 0\}$ and by $E(u)$ the affine hyperplane $H(u)+u$.



A body $K$ is origin symmetric if whenever $x \in K$, it follows that $-x \in K$. A body $K$ is centrally symmetric if a translate of $K$ is origin symmetric, i.e. if there is a vector $c \in \mathbb{E}^{n+1}$ such that $K-c$ is origin symmetric. 

Two convex bodies $M,N\subset \mathbb{E}^{n+1}$ are said to be \textit{congruent} if there exist an isometry $I:\mathbb{E}^{n+1} \rightarrow \mathbb{E}^{n+1}$ such that $I(M)=N$.

\textbf{Axis of symmetry, plane of symmetry.}
Let $L\subset \mathbb{E}^3$ be a line and let $\theta$ be an angle in 
$[0, 2\pi]$. We denote by 
$R_{(L,\theta)}:\mathbb{E}^3 \rightarrow \mathbb{E}^3$ the rotation with axis $L$ and with angle $\theta$. We just denote by $R_L$ and $R_{L,n}$ the maps $R_{(L,\pi)}$, $R_{(L, \frac{2\pi}{n})}$, respectively.

Let $K \subset \mathbb{E}^3$ be a convex body, $L\subset \mathbb{E}^3$ be a line and $n$ be an integer, $n \geq 2$.
The line $L$ is said to be an $n$-\textit{axis of symmetry} of $K$ if the following relation
\[
R_{L,n}(K) = K
\]
holds. In the case $n = 2$, a 2-axis of symmetry of the convex body K will just called axis of symmetry of $K$. For instance, if $C = [-1, 1] \times [-1, 1] \times [-1, 1]$ is a cube, centered at the origen, it has three types of axis of symmetry, namely,  the lines determined by the mid points of parallel edges not in the same face, the diagonals, and the lines determined by the centres of parallel faces, corresponding to the numbers 2,3 and 4.

Let $K \subset \mathbb{E}^3$ be a convex body and let
$\Pi$ be a plane. We denote by $S_{\Pi} :  \mathbb{E}^3 \rightarrow \mathbb{E}^3$ the reflection with respect
to $\Pi$. The plane $\Pi$ is said to be a \textit{plane of symmetry} of $K$ if the relation
\begin{eqnarray}\label{mami}
S_{\Pi} (K) = K
\end{eqnarray}
holds. 
  
\textbf{Field of congruent bodies.}
In order to prove the Theorem \ref{marisol} we need the following important definition. When $A$ is a $n$-dimensional convex body in the Euclidean space $\mathbb{E}^{n}$, a \textit{field of bodies congruent to} $A$ is a continuous function $A(u)$ defined for $u$ in the unit sphere $\mathbb{S}^n$, where $A(u)$ is a congruent copy of $A$ lying in a hyperplane of $\mathbb{E}^{n+1}$ perpendicular to $u$ and tangent to $\mathbb{S}^n$; here $A(u)$ is meant to be continuous in the Hausdorff metric. If additionally $A(u)-u= A(-u)+u$ for each $u$, we say $A(u)$ is a \textit{complete turning of} $A$ in $\mathbb{E}^{n+1}$. 

In \cite{had} Hadwiger proved that: \textit{No $n$-dimensional convex body with trivial group of symmetry}  \textit{can be completely turned in the euclidean $(n+1)$-space $\mathbb{E}^{n+1}$}. In \cite{mani} Mani showed that \textit{if $n\not=3,7,$ no $n$-dimensional convex body with finite group of symmetries determines a field of congruent convex bodies tangent to $\mathbb{S}^n$}. On the other hand, Mani also proved  that\textit{if $n$ is even only with spheres can be constructed fields of congruent convex bodies tangent to $\mathbb{S}^n$}. Furthermore, Mani claimed ``for odd dimensional spheres we have not found a general characterizations of bodies which can be turned around them". 

In \cite{burton} Burton made use of the methods of \cite{had} and \cite{mani} to prove that a \textit{3-dimensional convex body $M\subset \mathbb{E}^{4}$ can be completely turned in $\mathbb{E}^{4}$ if and only if $M$ is centrally symmetric}. In \cite{monte} Montejano, using methods of topology of fiber bundles,  generalized Burton's result: \textit{a convex body $M\subset \mathbb{E}^{n}$, $n \geq 2$, can be completely turned in $\mathbb{E}^{n+1}$ if and only if $M$ is centrally symmetric}. 

If all the $n$-dimensional sections of a $(n+1)$-dimensional convex body through a fixed inner point are congruent, clearly they give rise to a complete turning of some $n$-dimensional body in $\mathbb{E}^{n+1}$.

\section{Proof of Theorem \ref{marisol}}
  \textbf{The strategy of the proof.} The proof is structured in the following way:
\begin{itemize}
\item [a)] We will formalize the fact that the support cones of $M$ with apexes in $\bd K$ are affinely congruent in a continuous way.
\item [b)] Using the fact that the principal fibre bundle $$\tau: \mathbb{S}^2 \rightarrow O(3)$$ is not trivial (see the Appendix, special attention should be paid to relations i'), ii'), iii') and iv') and equations (\ref{pedro}), (\ref{infante}) and (\ref{negrete}), derived from the fact that the tangent bundle of $\mathbb{S}^2$ is trivial and compare these relations with i), ii), iii) and iv) and equations (\ref{lima}), (\ref{limonada}) and (\ref{fresa}) obtained in point b) ), we will prove that the cone $C_x$, $x\in K$, has either a plane of symmetry or an axis of symmetry.  
\item [c)] We will show that if the cone $C_x$, $x\in K$, has a plane of symmetry, then it is either a right circular cone or it has an axis of symmetry.
\item [d)] We will prove that if the cone $C_x$ has an axis of symmetry, then we can construct a field of congruent bodies and from here we conclude that the supporting cones $C_x$ of $M$, $x\in K$, are right circular cones. 
\item [e)] We will conclude that $M$ is a sphere using Matsuura's Theorem.
\end{itemize}
\textbf{a)} Notice that if $x_0,x\in \bd K$ and $\{x_n\} \subset \bd K$ are such that $x_n\rightarrow x$, when $n \rightarrow \infty$, since the support cones of $M$, with apexes in $\bd K$, are affinely congruent in a continuous way, there exist $\Omega \in O(3)$, a sequence 
$\{\Omega_n\}\subset O(3)$, a sequence $\{a_n\}\subset \mathbb{E}^3$ and maps $\Phi=a+\Omega$, $\Phi_n=a_n+\Omega_n$, $n=1,2,...$ with the property that 
\begin{eqnarray}\label{puchita}
\Phi(C_{x_0})=C_x, 
\end{eqnarray}
\begin{eqnarray}\label{puchis}
\Phi_n(C_{x_0})=C_{x_n},
\end{eqnarray}
$n=1,2...$ and 
\begin{eqnarray}\label{limon}
\Phi_{n} \rightarrow \Phi,
\end{eqnarray} 
when $n \rightarrow \infty$ (in the subspace topology from $\mathbb{E}^{(n+1)^2}$). Thus
\begin{eqnarray}\label{pompita}
C_{x_n} \rightarrow C_x
\end{eqnarray}
when $n \rightarrow \infty$ (in the Hausdorff metric) and
\begin{eqnarray}\label{pozole}
\Omega_{n} \rightarrow \Omega.
\end{eqnarray} 
when $n \rightarrow \infty$ (in the subspace topology from $\mathbb{E}^{(n+1)^2}$). Indeed, by (\ref{puchita}) $x= a+\Omega(x_0)$, that is, $a=x-\Omega(x_0)$. Analogously, by (\ref{puchis}) $x_n= a_n+\Omega_n(x_0)$, $n=1,2,...$, hence $a_n=x_n-\Omega_n(x_0)$. Therefore
$\Phi=(x-\Omega(x_0))+\Omega$, $\Phi_n=(x_n-\Omega_n(x_0))+\Omega_n$, $n=1,2,...$ From these relations and (\ref{limon}) it follows that (\ref{pozole}) holds.

We choose the origen of a system of coordinates at an interior point of $M$. We define the continuous map $\phi:\bd K \rightarrow \mathbb{S}^2$ as the inverse of the central projection $x\mapsto u:=\frac{x}{\|x\|}$. From now on we will fix a point $x_0\in \bd K$. In terms of the function $\phi$ the last paragraph can be re written in the following way. Let $x\in \bd K$ and let $\{x_n\}\subset \bd K$ be a sequence such that $x_n\rightarrow x$ when $n \rightarrow \infty$. Let $u:=\frac{x}{\|x\|}$ and $u_n:=\frac{x_n}{\|x_n\|}$. By the hypothesis there exist $\Omega_{u}\in O(3)$, a sequence  $\{\Omega_{u_n}\} \subset O(3)$ such that maps $\Phi_x=a_x+\Omega_{u}$, $\Phi_{x_n}=a_{x_n}+\Omega_{u_n}$, $n=1,2..$ have the property that (\ref{puchita}), (\ref{puchis}) holds for the maps $\Phi_x$ and $\Phi_{x_n}$, respectively. We also have that (\ref{limon}) holds for $\Phi_x$ and $\Phi_{x_n}$ and, finally, for the maps $\Omega_{u}$ and  
$\Omega_{u_n}$ the relation (\ref{pozole}) holds. 

\textbf{b)} Since the fibre bundle $\tau: \mathbb{S}^2 \rightarrow O(3)$ is not a product, it does not have a section (see the Appendix). Thus there exists 
\begin{itemize}
\item [i)] $x^* \in \bd K$ and sequences $\{x_n\}, \{z_n\}\subset \bd K$    such that $x_n\rightarrow x^*$ and  $z_n\rightarrow x^*$
when $n \rightarrow \infty$. 
\item [ii)] $\Omega_{u^*}\in O(3)$, a sequence  $\{\Omega_{u_n}\} \subset O(3)$ such that maps 
$$\Phi_{x^*}=a_{x^*}+\Omega_{u^*}, \textrm{   }\textrm{  }\Phi_{x_n}=a_{x_n}+\Omega_{u_n}, \textrm{   } \textrm{   }n=1,2,...$$ 
have the property that (\ref{puchita}), (\ref{puchis}) holds for the maps $\Phi_{x^*}$ and $\Phi_{x_n}$, respectively. We also have that (\ref{limon}) holds for $\Phi_{x^*}$ and $\Phi_{x_n}$. Furthermore, in particular, we have that
\begin{eqnarray}\label{lima}
\Omega_{u_n} \rightarrow \Omega_{u^*}
\end{eqnarray}
when $n\rightarrow \infty$ (in the subspace topology from $\mathbb{E}^{(n+1)^2}$), where $u^*:=\frac{x^*}{\|x\|}$ and $u_n$ is defined as before (see Fig. \ref{baker}). 
 \item [iii)] $\bar{\Omega}_{u^*}\in O(3)$, a sequence  $\{\Omega_{v_n}\} \subset O(3)$ such that maps 
$$\bar{\Phi}_{x^*}=\bar{a}_{x^*}+\bar{\Omega}_{u^*}, \textrm{  }\textrm{  }\Phi_{z_n}=a_{z_n}+\Omega_{v_n}, \textrm{  }\textrm{  } n=1,2,..$$
have the property that (\ref{puchita}), (\ref{puchis}) holds for the maps $\bar{\Phi}_{x^*}$ and $\Phi_{z_n}$, respectively. We also have that  and (\ref{limon}) holds for $\bar{\Phi}_{x^*}$ and $\Phi_{z_n}$. Furthermore, in particular, we have that
\begin{eqnarray}\label{limonada}
\Omega_{v_n} \rightarrow \bar{\Omega}_{u^*}
\end{eqnarray} 
when $n\rightarrow \infty$  (in the subspace topology from $\mathbb{E}^{(n+1)^2}$),  where $v_n:=\frac{z_n}{\|x_n\|}$ (notice that $v_n \rightarrow u^*$ when $n\rightarrow \infty$, see Fig. \ref{baker}).  
\item [iv)] The orthogonal maps $\Omega_{u^*}$ and $\bar{\Omega}_{u^*}$ are such that $\Omega_{u^*} \not=\bar{\Omega}_{u^*}$, i.e.,
\begin{eqnarray}\label{fresa}
\Omega_{u^*} ^{-1} \circ  \bar{\Omega}_{u^*}\not=\textrm{id}.
\end{eqnarray} 
\end{itemize}
\begin{figure}[H]
    \centering
    \includegraphics[width=.73\textwidth]{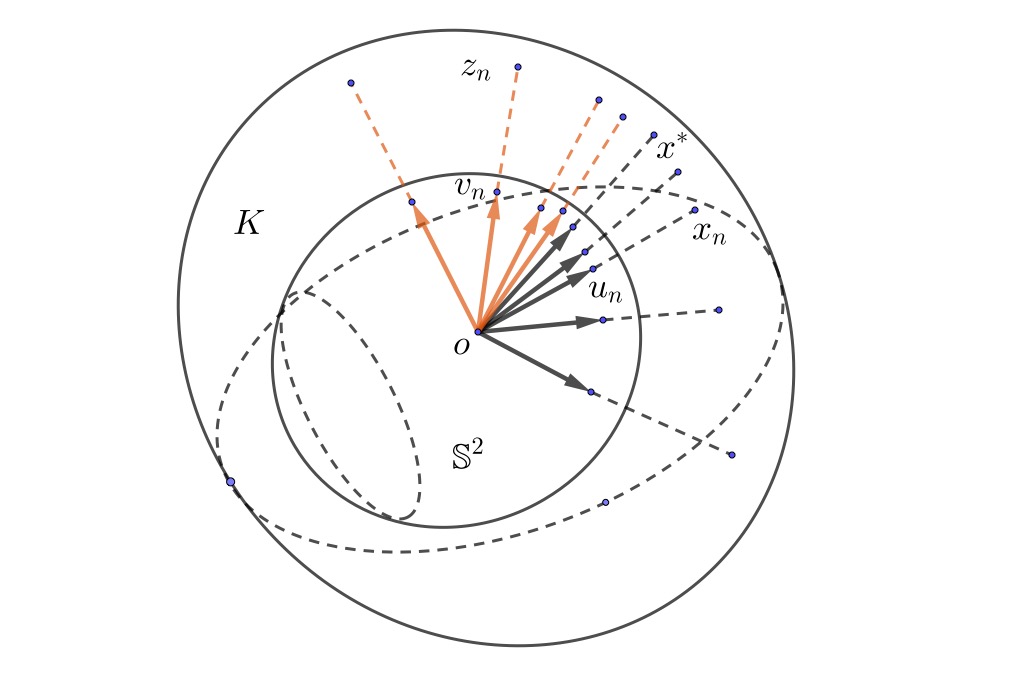}
    \caption{The sequence $\{u_n\}$, $\{v_n\}$ converges to $u^*$ when $n\rightarrow \infty$.}
    \label{baker}
\end{figure}
Notice that (\ref{lima}) and (\ref{limonada}) signifies that we have two ways to converge to $u^*$, in $\mathbb{S}^2$, and the corresponding sequence of elements of $O(3)$  converges to two different maps, it was stablished in iv).
 
By (\ref{limon}) applied to $\Phi_{x_n}$ and by (\ref{puchita}) applied to $\Phi_{x^*}$ it follows that
\begin{eqnarray}\label{oasis}
\Phi_{x_n}(C_{x_0})\rightarrow \Phi_{x^*}(C_{x_0})=C_{x^*}
\end{eqnarray}
when $n\rightarrow \infty$. On the other hand, by (\ref{limon}) applied to $\Phi_{z_n}$ and by (\ref{puchita}) applied to $\bar{\Phi}_{x^*}$ it follows that
\begin{eqnarray}\label{bolotas}
\Phi_{z_n}(C_{x_0})\rightarrow \bar{\Phi}_{x^*} (C_{x_0})=C_{x^*}
\end{eqnarray}
when $n\rightarrow \infty$. From (\ref{oasis}) and (\ref{bolotas}) it follows that 
\begin{eqnarray}\label{loca}
\Phi_{x^*}(C_{x_0})=\bar{\Phi}_{x^*} (C_{x_0}),
\end{eqnarray}
Thus
\begin{eqnarray}\label{nina}
\Omega_{u^*}(C_{x_0})=\bar{\Omega}_{u^*} (C_{x_0}),
\end{eqnarray}
i.e.,     
\begin{eqnarray}\label{simetrico}
C_{x_0}=(\Omega_{u^*}^{-1} \circ \bar{\Omega}_{u^*}) (C_{x_0}).
\end{eqnarray} 
From (\ref{simetrico}), by virtue of (\ref{fresa}), the cone $C_{x_0}$ has either an axis of symmetry or a plane of symmetry. 

\textbf{c)} First we suppose that the cone $C_{x_0}$ has a plane of symmetry, say 
$\Pi$. By (\ref{loca}), $\Phi_{x^*}(\Pi)$ and $\bar{\Phi}_{x^*}(\Pi)$ are planes of symmetry of $C_{x^*}$. By iv),  $\Phi_{x^*} \not=\bar{\Phi}_{x^*} $,  consequently, either $\Phi_{x^*}(\Pi)\not= \bar{\Phi}_{x^*}(\Pi)$ or one of the maps $\Phi_{x^*}$, $\bar{\Phi}_{x^*}$ is a reflection with respect to $\Pi$ and the other is a rotation with respect to an axis contained in $\Pi$ by an angle $\pi$, i.e., the cone $C_{x_0}$ has an axis of symmetry. Thus now we assume that $\Phi_{x^*}(\Pi)\not= \bar{\Phi}_{x^*}(\Pi)$. Then either $C_{x_0}$ has an infinite number of planes of symmetry or a finite number of planes of symmetry. In the first case, $C_{x_0}$ is a straight circular cone. In the later case, $C_{x_0}$ has an axis of symmetry.

\textbf{d)} Now we suppose that the cone $C_{x_0}$ has an axis of symmetry, say $L_{x_0}$.  Let $\Delta_{x_0}$ be a plane perpendicular to $L_{x_0}$ which is at distance equal to 1 from $x_0$ and such that $M_{x_{0}}:=C_{x_0}\cap \Delta_{x_0}\not=\emptyset$. Then, for $x\in \bd K$, we define 
$\Delta_x:=\Phi_x(\Delta_{x_0})$ and $M_{\eta(x)}:=\Phi_x(M_{x_{0}})$, where $\Phi_x$ satisfies (\ref{puchita}), and  
  $$\eta: \bd K\to \mathbb S^2, \quad \mbox{is the following map: }$$
 $\eta(x)$ is the normal vector to $\Delta_x$ pointing out to $x$ (notice that $\eta(x)$ is parallel to the axis of $C_{x}$).
  
\begin{lemma}\label{homeo}
The  map $\eta: \bd K\to \mathbb S^2$ is a homeomorphism.
\end{lemma}
\begin{proof} 
First, we are going to prove that $\eta: \bd K\to S^2$ is a continuous map. 
We will prove that if $x\in \bd K$ and $\{x_n\}\subset \bd K$ is a sequence such that $x_n\rightarrow x$ when $n \rightarrow \infty$, then $\eta(x_n)\rightarrow \eta(x)$ when $n \rightarrow \infty$. By (\ref{pompita}), the axis of symmetry $L_{x_n}$ of $C_{x_n}$ tends to the axis of symmetry $L_x$ of $C_{x}$ when $n \rightarrow \infty$. Thus  
$\Delta_{x_n}\rightarrow \Delta_x$ when $n \rightarrow \infty$, i.e., $\eta(x_n)\rightarrow \eta(x)$ when $n \rightarrow \infty$, i.e., $\eta$ is continuous.

Now we will prove that $\eta$ is injective. Suppose that $x\not=x' \in \bd K$ and $\Delta_x$, $\Delta_x'$ are parallel (never the same, otherwise, $x=x'$). First, we assume that $x'$ lies in the interior of $C_x$. It follows that $C_{x'}\subset C_x$. Thus $S_x\subset \mathbb{E}^3 \backslash S_{x'}$. This implies that $S_x\cap \bd M=\emptyset$  but this contradicts that $S_x$ is the support cone of $M$ with apex $x$. 
 
Now we assume that $x'\in \mathbb{E}^3 \backslash C_x$. Then there exists a support plane $\Gamma$ of $C_x$ which separates $M$ and 
$x'$. The plane $\Gamma'$, parallel to $\Gamma$ and passing through $x'$, is supporting plane of $C_{x'}$. Hence there is a point of $M$ in $\Gamma'$ but this contradicts that $\Gamma$ separates $M$ and 
$x'$.

Since $\bd K$ is compact, this implies that,  
  $\eta: \bd K\to \eta(\bd K)$  is a homeomorphism and consequently, $H_{2}(\eta(\bd K), \mathbb Z)=\mathbb Z$. On the other hand, any proper subset of $ \mathbb S^2$ has $2$-dimensional homology equal to zero. Hence  $\eta(\bd K)= \mathbb S^2$, as desired.
\end{proof}

\textbf{Proof of Theorem \ref{marisol}.} If $C_{x_0}$ is a straight circular cone then all the support cones $C_x$ of $M$, $x\in \bd K$, are straight circular cones. By the theorem of Matsuura \cite{Matsuura}, $M$ is a sphere.

Now we suppose that the cone $C_{x_0}$ has an axis of symmetry.
By virtue of the Lemma \ref{homeo}, the family $\{M_{\eta(x)}\}$ is a field of congruent convex bodies. By Mani's  Proposition 1 of \cite{mani}, $M_{x_{0}}$ is a circle. Thus for every $x\in \bd K$, the cone $C_x$ is 
is a straight circular cone. By the theorem of Matsuura \cite{Matsuura}, $M$ is a sphere.
\section{Apendix} 
Let $O(n+1)$ be the group of isometries of $\mathbb{E}^{n+1}$. $O(n)$ is a closed subgroup of $O(n+1)$ and the quotient space
$$
O(n+1)/O(n) \textrm{ is homeomorphic to } \mathbb{S}^n.
$$
Furthermore, the quotient map
$$
\tau:O(n+1) \rightarrow \mathbb{S}^n
$$
is a principal bundle with fibre $O(n)$ (see Section 7.6 of \cite{steenrod}). In fact, it is the associated principal bundle of the tangent vector bundle of $\mathbb{S}^n$.

As a corollary we have that two fibre bundles with the same base space and the same fibre $F$ are equal if and only if their associated principal bundle are equal. The next results are contained in \cite{steenrod} (see 8.5 the cross-section Theorem pag. 36]).

\textbf{Theorem 15.7.1.} \textit{A principal bundle with group structure
$\mathcal{G}$ is trivial or the product bundle if and only if it admits a cross section.}

\textbf{Corollary 15.7.1.} \textit{A fiber bundle with group structure $\mathcal{G}$ is the trivial bundle if and only if its associated principal bundle admits a cross section}.

\textbf{Remark 15.7.1.} \textit{A fibre bundle may admit a cross section without being trivial}.
 
The following result was taken from Section 27.4 of \cite{steenrod}.
 
\textbf{Theorem NS.} \textit{The tangent bundles of $\mathbb{S}^1$, $\mathbb{S}^3$, and $\mathbb{S}^7$ are equivalent to product bundles. If $n$ is even, or if $n> 1$ and 
$n= 1\mod 4$, then the tangent bundle of $S^n$ is not equivalent to a product bundle.} 
 
\textbf{Claim.} The principal bundle $\tau:O(3) \rightarrow \mathbb{S}^2$ is not trivial.  

\textit{Proof.} On the contrary, let us assume that the principal bundle $q:O(3) \rightarrow \mathbb{S}^2$ is trivial.  By Theorem 15.7.1. it admits a cross section. Then by Corollary 15.7.1. the tangent vector bundle of $\mathbb{S}^2$ is the trivial bundle. However this contradicts Theorem NS. Thus  $q:O(3) \rightarrow \mathbb{S}^2$ is not trivial.

\textbf{Vector fields on the sphere.}
Vector fields on the sphere with a constant norm have special properties, and the existence and nature of such fields depend on the sphere's dimension. For the unit sphere, particularly odd-dimensional spheres, specific vector fields are known, such as the Hopf vector fields, which are constant-length vector fields tangent to the Hopf fibration (see pag. 106 of \cite{steenrod}) and are also smooth minimal vector fields. On even-dimensional spheres like $\mathbb{S}^2$, the condition of constant norm is more restrictive. 

A vector field on a sphere can have a constant, non-zero norm only if the dimension of the sphere is odd. This is a consequence of the Hairy Ball Theorem \cite{Eisenberg}, which states that any continuous tangent vector field on an even-dimensional sphere must have at least one point where the vector is zero.  
 
The last paragraph, stated in a more specific way for the case of $\mathbb{S}^2$, means that if in the fibre bundle
 $$
\mu: T\mathbb{S}^2 \rightarrow \mathbb{S}^2
$$ 
we define a vector field (a section) $s: \mathbb{S}^2 \rightarrow T\mathbb{S}^2 $ with constant norm there necessarily exists a point
\begin{itemize}
\item [i')] $u^* \in \mathbb{S}^2$ and sequences $\{u_n\}, \{v_n\}\subset \mathbb{S}^2$    such that $u_n\rightarrow u^*$, $v_n\rightarrow u^*$,
when $n \rightarrow \infty$, 
\item [ii')]  and  
\begin{eqnarray}\label{pedro}
s(u_n) \rightarrow s^*,
\end{eqnarray}
 \item [iii')]  
\begin{eqnarray}\label{infante}
s(v_n) \rightarrow \bar{s}^*
\end{eqnarray} 
when $n\rightarrow \infty$,  
\item [iv')] where the vector $s^*$ and $\bar{s}^*$ are tangent vector of $\mathbb{S}^2$ at the point $u^*$ and 
\begin{eqnarray}\label{negrete}
s^*\not=\bar{s}^*.
\end{eqnarray} 
\end{itemize}

\section{Declarations}
Availability of data and material. 
Not applicable. 

Competing interests. 
Not applicable.

\end{document}